\documentclass[12pt]{article}

\usepackage{amsmath,amsthm,amsfonts,amssymb}
\usepackage{graphicx,latexsym}
\usepackage{url}
\newtheorem{thm}{Theorem}[section]
\newtheorem{cor}[thm]{Corollary}

\newcommand{\cp}{\,\square\,}

\begin{document}

\title{A Lower Bound and Several Exact Results on the $d$-Lucky Number}

\author{
	Sandi Klav\v zar $^{a,b,c}$
	\and
		Indra Rajasingh $^{d}$
	\and
	D. Ahima Emilet $^{d}$ 
}

\date{\today}

\maketitle
\begin{center}
$^d$ Faculty of Mathematics and Physics, University of Ljubljana, Slovenia\\
{\tt sandi.klavzar@fmf.uni-lj.si}
\medskip

$^b$ Faculty of Natural Sciences and Mathematics, University of Maribor, Slovenia\\
\medskip

$^c$ Institute of Mathematics, Physics and Mechanics, Ljubljana, Slovenia\\
\medskip

$^d$ School of Advanced Sciences, Vellore Institute of Techonology, Chennai-600127, India\\
{\tt indra.rajasingh@vit.ac.in}
\end{center}

\begin{abstract}
If $\ell: V(G)\rightarrow {\mathbb N}$ is a vertex labeling of a graph $G = (V(G), E(G))$, then the $d$-lucky sum of a vertex $u\in V(G)$ is  $d_\ell(u) = d_G(u) + \sum_{v\in N(u)}\ell(v)$. The labeling $\ell$ is a $d$-lucky labeling if $d_\ell(u)\neq d_\ell(v)$ for every $uv\in E(G)$. The $d$-lucky number $\eta_{dl}(G)$ of $G$ is the least positive integer $k$ such that $G$ has a $d$-lucky labeling $V(G)\rightarrow [k]$. A general lower bound on the $d$-lucky number of a graph in terms of its clique number and related degree invariants is proved. The bound is sharp as demonstrated with an infinite family of corona graphs. The $d$-lucky number is also determined for the so-called $G_{n,m}$-web graphs and graphs obtained by attaching the same number of pendant vertices to the vertices of a generalized cocktail-party graph.  
\end{abstract}


\medskip\noindent
{\bf AMS Subj.\ Class:} 05C78

\section{Introduction}
\label{sec:intro}

In the celebrated paper~\cite{ka2004}, Karo\'{n}ski, \L uczak, and Thomason asked whether the edges of any graph with no component $K_2$ can be assigned weights from $\{1,2,3\}$ so that adjacent vertices have different sums of incident edge weights, in other words, such that the resultant vertex weighting is a proper coloring. Although the paper mentions no ``conjecture", the question turned later into the 1-2-3 Conjecture. The progress on the conjecture until 2012 has been surveyed in~\cite{seamone-2012}, while for recent progress see~\cite{hornak-2018, kalkowski-2017} and references therein. 

The paper~\cite{ka2004} can also be seen as the seed for the investigation of other types of graph labelings in which integers are assigned to some elements of the graph (vertices, edges, or both of them), such that the labeling yields a proper vertex coloring. For instance, Czerwi\'{n}ski, Grytczuk, and \.{Z}elazny~\cite{cz2009} introduced the concept of the lucky labeling and proposed the conjecture $\eta(G)\leq \chi(G)$, where $\eta(G)$ is the lucky number of $G$ (and, of course, $\chi(G)$ is the chromatic number of $G$). For more information on the lucky labelings see~\cite{ah2012, ak2013}. Similar to the lucky number, Chartrand, Okamoto, and Zhang~\cite{Gary} introduced sigma colorings, where the value at a vertex is obtained as the sum of the weights in its neighborhood. Club scheduling problems and hospital planning are real life applications of sigma colorings, cf.~\cite{lagura-2015}. For additional related labelings we refer to~\cite{de2013}. We mention in passing that graph labelings have a variety of applications such as incorporating redundancy in disks, designing drilling machines, creating layouts for circuit boards, and configuring resistor networks, see~\cite{wang2014}. Finally, different graph labelings were and are still extensively investigated, we refer to the recent developments~\cite{baca-2018, lozano-2019}. 

In this paper we are interested in $d$-lucky labelings that were introduced by Miller et al.~\cite{ahi2015} as a variant of the lucky labelings as follows. Let $N(u)=\left\{v\in V(G):\ uv\in E(G)\right\}$ be the open neighborhood of a vertex $u$ in a graph $G$. If $\ell: V(G)\rightarrow {\mathbb N}$ is a vertex labeling, then the {\em $d$-lucky sum} of a vertex $u\in V(G)$ with respect to $\ell$ is 
$$d_\ell(u) = d_G(u) + \sum_{v\in N(u)}\ell(v)\,,$$ 
where $d_G(u)$ is the degree of $u$. The labeling $\ell$ is a {\em $d$-lucky labeling} if $d_\ell(u)\neq d_\ell(v)$ holds for every pair of adjacent vertices $u$ and $v$. The {\em $d$-lucky number} $\eta_{dl}(G)$ of $G$ is the least positive integer $k$ such that $G$ admits a $d$-lucky labeling $\ell: V(G)\rightarrow [k] = \{1,\ldots,k\}$. Lucky labelings are obtained from $d$-lucky labelings by omitting the additive term $d_G(u)$. A closely related concept of the adjacent vertex distinguishing colorings is defined analogously, except that one adds up the labels in the closed neighborhood of a vertex, see~\cite{axenovic-2016, dehghan-2017}.

In the next section we prove a general lower bound on the $d$-lucky number of a graph in terms of its clique number and related degree invariants. The bound is sharp as demonstrated with an infinite family of corona graphs. The latter result is in turn used in Section~\ref{sec:web-graphs} to determine the $d$-lucky number of the so called $G_{n,m}$-web graphs. We conclude the paper with the $d$-lucky number of graphs obtained by attaching the same number of pendant vertices to the vertices of a generalized cocktail-party graph.  

\section{A Lower bound on the $d$-lucky number}
\label{sec:lower}

In this section we give a lower bound on $\eta_{dl}(G)$ of a graph $G$ in terms of its clique number $\omega(G)$ (that is, the size of a largest complete subgraph) and demonstrate that the bound is sharp.  

To state the main result we need the following notation. If $Q$ is a clique of $G$, then let $\delta_G(Q)$ and $\Delta_G(Q)$ be the minimum and the maximum degree in $G$ among the vertices from $Q$, respectively. Let further ${\cal Q}(G)$ be the set of largest cliques of $G$. Then we have:  

\begin{thm}
\label{thm:lower}
If $G$ is a connected graph, then 
$$\eta_{dl}(G) \ge \max_{Q\in {\cal Q}(G)}\left\lceil \frac{2\delta_G(Q) - \Delta_G(Q)+1}{\Delta_G(Q) - \omega(G) + 2}\right\rceil\,.$$
\end{thm}

\begin{proof}
Let $\omega(G) = s$ and let $Q\in {\cal Q}(G)$, so that $|n(Q)| = s$. Let $\eta_{dl}(G) = k$ and let $\ell: V(G) \rightarrow [k]$ be a $d$-lucky labeling of $G$. Set 
$$x = \sum_{u\in V(Q)} \ell(u)\,.$$
If $u\in V(Q)$ and $\ell(u) = i\in [k]$, then  
$$(x-i) + (d_G(u)-(s-1)) + d_G(u) \le d_\ell(u) \le (x-i) + k(d_G(u)-(s-1)) + d_G(u)\,.$$ 
Since $\ell(u) \in [k]$, the largest possible value of $d_\ell(u)$ is $(x-1) + k(\Delta_G(Q)-s+1) + \Delta_G(Q)$, and the smallest possible value of $d_\ell(u)$ is $(x-k) + (\delta_G(Q)-s+1) + \delta_G(Q)$. Therefore, vertices from $Q$ receive at most 
$$N = [(x-1) + k(\Delta_G(Q)-s+1) + \Delta_G(Q)] - [(x-k) + (\delta_G(Q)-s+1) + \delta_G(Q)] + 1$$ 
distinct $d_\ell$-sum values. Since the vertices of $Q$ receive pairwise different labels, $N\ge s$ must hold. From this inequality a straightforward computation yields 
$$k(\Delta_G(Q)-s+2) \ge 2\delta_G(Q) - \Delta_G(Q)+1\,.$$ 
The assertion now follows because this inequality holds true for any clique from ${\cal Q}$ and since the $d$-lucky number is an integer. 
\end{proof}

Theorem~\ref{thm:lower} significantly simplifies for certain classes of graphs, an instance is presented in the next result. 

\begin{cor}
\label{cor:lower-regular}
If $G$ is a connected graph and every vertex $v$ in any largest clique of $G$ has $deg_{G}(v) = r$, then $$\eta_{dl}(G) \ge \left\lceil \frac{r + 1}{r - \omega(G) + 2}\right\rceil\,.$$
\end{cor}

To see that Corollary~\ref{cor:lower-regular} (and hence also Theorem~\ref{thm:lower}) is sharp, consider complete graphs $K_n$ which are $(n-1)$-regular with $\omega(K_n) = n$. Thus Corollary~\ref{cor:lower-regular} implies $\eta_{dl}(K_n) \ge n$. To find another (non-trivial) family of such graphs recall that the {\em corona} $G\circ H$ of graphs $G$ and $H$ is defined as follows. Let $V(G) = [n]$ and let $H$ be a graph. Then $G\circ H$ is obtained from the disjoint union of $G$ and $n$ disjoint copies $H_1, \ldots, H_n$ of the graph $H$, where the the vertex $i\in V(G)$ is connected with an edge to every vertex of $H_i$. Let further $\overline{G}$ denote the complement of $G$. Then we have: 

\begin{thm} 
\label{thm:corona}
If $n\ge 2$ and $r\ge 1$, then
$$\eta_{_{dl}}(K_n\circ \overline{K_r}) = \left\lceil \frac{n+r}{r+1}\right\rceil\,.$$
\end{thm}

\begin{proof} 
To shorten the notation, set $G_{n,r} = K_n\circ \overline{K_r}$. Clearly, $\omega(G_{n,r}) = n$ and every vertex $v$ in the largest clique $K_{n}$ of $G_{n,r}$ has $d_{_{G_{n,r}}}(v) = n - 1 + r$. Hence Corollary~\ref{cor:lower-regular} gives $\eta_{_{dl}}(K_n\circ \overline{K_r}) \geq  \left\lceil \frac{n+r}{r+1}\right\rceil$. 

To prove the reverse inequality, let $X = \{v_1, \ldots, v_n\}$ be the vertex set of the subgraph $K_n$ of $G_{n,r}$ and let $X_i$, $i\in [n]$, be the set of pendant vertices of $G_{n,r}$ adjacent to $v_i$.

Suppose first that $n\le r+1$. In this case label all the vertices of $X$ with $1$. Further, if $i\in [n]$, then label $(i-1)$ vertices from $X_i$ with $2$ and the other $r-i+1$ vertices of $X_i$ with label $1$, See Fig.~\ref{k5,4_and_k11,1}(a). As $n\le r+1$, the sums of the labels in $X_i$ and $X_j$ are different for all $i\ne j$. Hence, this labeling is a $d$-lucky labeling with two labels and so $\eta_{_{dl}}(K_n\circ \overline{K_r}) \le 2$. We observe that as $n \geq 2$, 
$$\frac{n+r}{r+1} \geq \frac{2+r}{r+1} > 1$$ 
and as $n \leq r + 1$, 
$$\frac{n+r}{r+1} \leq \frac{2r+1}{r+1} = 1 + \frac{r}{r + 1} < 2\,.$$ 
Therefore $\left\lceil \frac{n+r}{r+1}\right\rceil = 2$. Thus $\eta_{_{dl}}(K_n\circ \overline{K_r}) \leq \left\lceil \frac{n+r}{r+1}\right\rceil$.

\begin{figure}[ht!]
\centering
\includegraphics[scale=0.79]{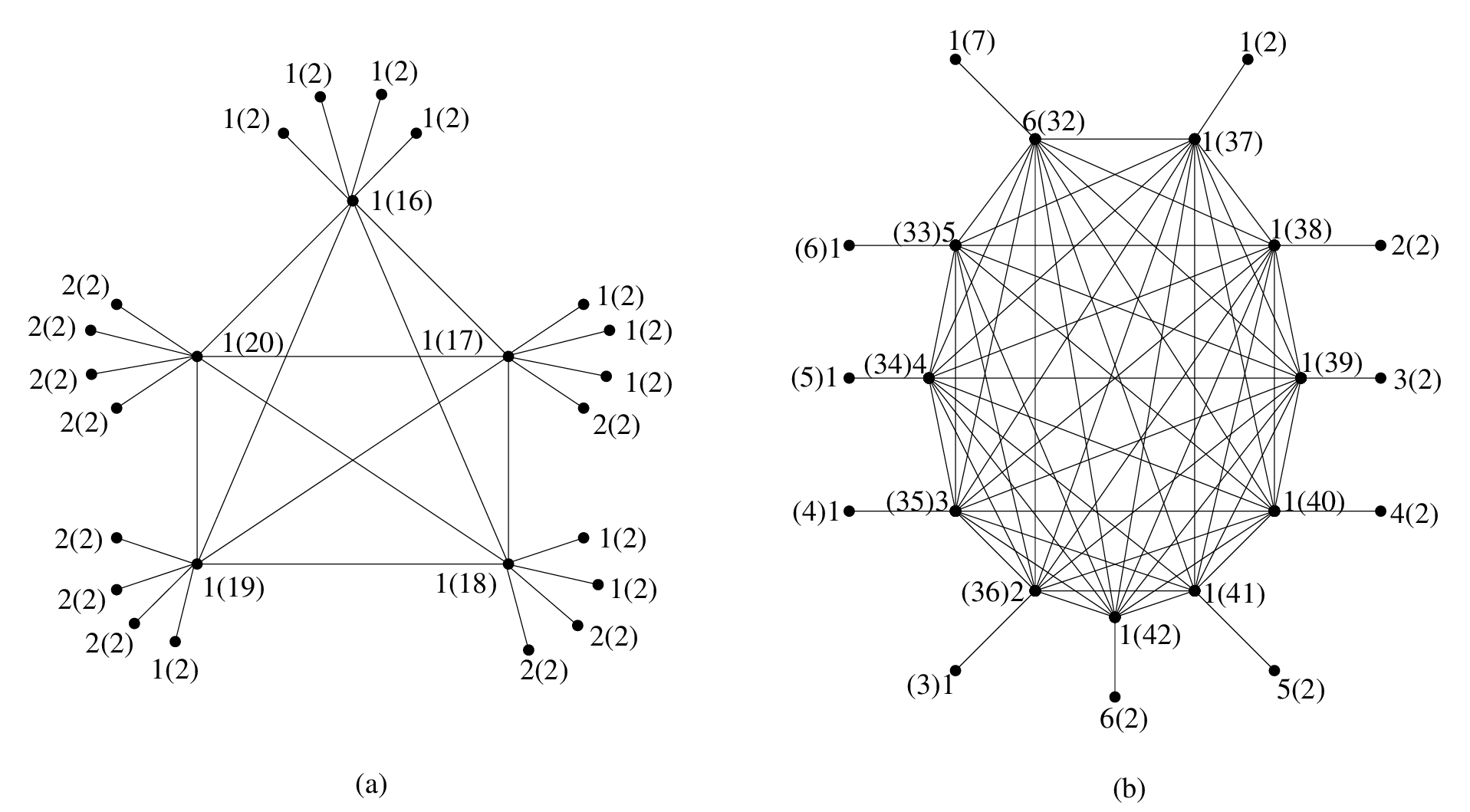} 
\caption{~(a)~$d$-lucky labeling of $G_{5,4}$\ (b) $d$-lucky labeling of $G_{11,1}$}
\label{k5,4_and_k11,1}
\end{figure}

Assume now that $n > r + 1$. In this case label the vertices $v_1, v_2, \ldots, v_{kr - r + 1}$ of $X$ with $1$ and the remaining $n - \left(kr - r + 1\right)$ vertices in $X$ with labels $2,3,\ldots, n - \left(kr - r + 1\right) + 1$, where $k = \left\lceil \frac{n+r}{r+1}\right\rceil$. Label vertices in $X_{_{i}}$, $1\leq i \leq kr - r + 1$ with labels such that the sums $x_{_{i}}$'s, $1\leq i \leq kr - r + 1$ are all distinct lying between $r$ and $k r$. Further, label all other pendant vertices as $1$. See Figure $\ref{k5,4_and_k11,1}(b)$. The contribution to the $d$-sums of the $k r - r + 1$ vertices of $X$ which are labeled $1$ are, respectively, 
$$x-1+r, x-1+\left(r+1\right), \ldots, x-1+ kr\,$$ 
and the contribution to the $d$-sums of the remaining $n - \left(kr - r + 1\right)$ vertices of $X$ are, respectively, 
$$x-2+r, x-3+r, \ldots, x-\left(n- kr+r\right)+r\,.$$ 
Since the degree of every vertex in $X$ is $t = n - 1 +r$, the $d$-sums of the vertices of $X$ are $t+x+r-\left(n- kr+r\right),~t+x+r-\left(n- kr+r\right)+1,~\ldots, t+x+r-3,~t+x+r-2,~t+x+r-1,~t+x+r,~t+x+r+1, \ldots,~t+x+r+\left(kr-r-1\right)$, which are consecutive integers between $t+x+ kr-n$ and $t+x+ kr-1$. Thus all vertices of $X$ receive distinct consecutive $d$-sums. Since the pendant vertices form an independent set, the labels of these vertices do not contribute to the $d$-lucky number. As $n > r+1$, the sums of the labels in $X_i$ and $X_j$ are different for all $i\ne j$. Therefore, 
\begin{eqnarray*}
\eta_{_{dl}}(K_n\circ \overline{K_r})  & \leq & n - (kr - r + 1) + 1 = (n + r) - \left\lceil \frac{n+r}{r+1}\right\rceil r \\
& \leq & (n + r) - \left(\frac{n+r}{r+1} \right)r \leq \frac{n+r}{r+1}  \leq  \left\lceil \frac{n+r}{r+1}\right\rceil\,. 
\end{eqnarray*}
We conclude that $\eta_{_{dl}}(K_n\circ \overline{K_r}) =  \left\lceil \frac{n+r}{r+1}\right\rceil$.
\end{proof}

Theorem~\ref{thm:corona} thus gives an infinite, non-trivial family of graphs for which the equality is achieved in Theorem~\ref{thm:lower}.  

\section{More exact $d$-lucky numbers}
\label{sec:web-graphs}

In this section we determine the $d$-lucky number of two infinite families of graphs. With the aid of Theorem~\ref{thm:corona} we obtain the $d$-lucky number of the $G_{_{n,m}}$-web graphs defined below. At the end of the section we then present the $d$-lucky number of graphs obtained by attaching the same number of pendant vertices to the vertices of a generalized cocktail-party graph.  

For $m,n\geq 3$, set $C_{m,n} = P_{_{m}}\cp C_{_{n}}$, where $\cp$ denotes the standard Cartesian product of graphs~\cite{imrich-2008}. The graphs $C_{m,n}$ are sometimes called {\em cylinders}. Let us call the edges of $C_{m,n}$ that project on $P_{_{m}}$ {\em radial edges} of $C_{m,n}$, and the other edges (that is, those that project on $C_n$) {\em cycle edges}. Further, the two $C_n$-layers whose vertices are of degree $3$ will be called the {\em top layer} and the {\em bottom layer}, respectively, while the other $C_n$-layers will be referred to as layer$1$, $\ldots$, layer $m - 1$, see Fig.~$\ref{k6,3web}(a)$. Now, the {\em $G_{m,n}$-web graph} is obtained from the disjoint union of $C_{m,n}$ and $K_n$ by adding a matching between the top layer of $C_{m,n}$ and the vertices of $K_n$, subdividing each of these matching edges, and subdividing all the edges of the $C_m$-layers of $C_{m,n}$. See Fig.~$\ref{k6,3web}(b)$ and Fig.~$\ref{k6,3web}(c)$ for $G_{3,6}$ and $G_{4,6}$, respectively.  

\begin{figure}[ht!]
\centering
\includegraphics[scale=0.69]{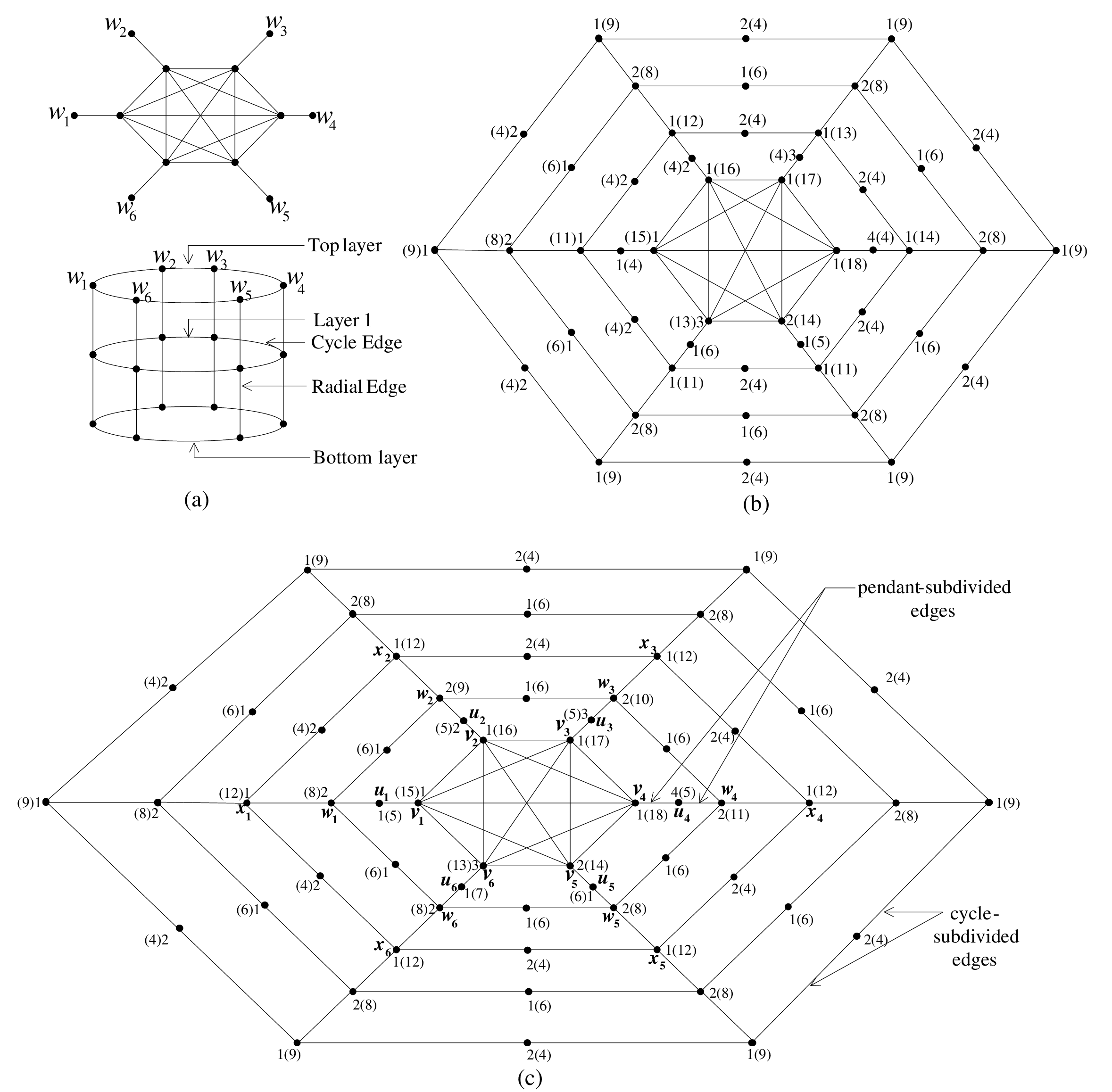} 
\caption{(a) $P_{3}\cp C_{6}$\ (b) $d$-lucky labeling of the $G_{_{3,6}}$-web graph\ (c) $d$-lucky labeling of the $G_{_{4,6}}$-web graph}
\label{k6,3web}
\end{figure}

The edges of $G_{m,n}$ obtained by subdividing the matching edges between $C_{m,n}$ and $K_n$ will be called {\em pendant-subdivided edges}, and the edges obtained by subdividing the $C_m$-layers {\em cycle-subdivided edges}, see Fig.~$\ref{k6,3web}(c)$ again. The result of this section now reads as follows. 

\begin{thm} 
\label{thm:web}
If $m\ge 3$ and $n\ge 5$, then  $\eta_{_{dl}}(G_{m.n}) = \left\lceil \frac{n+1}{2}\right\rceil$. 
\end{thm}

\begin{proof}
To simplify the notation, set $G = G_{m,n}$ for the rest of the proof. Note first that $K_n$ is the unique largest clique of $G$, its vertices being of degree $n$ in $G$, hence by Corollary~\ref{cor:lower-regular}, $\eta_{_{dl}}(G) \ge \left\lceil \frac{n+1}{2}\right\rceil$. It thus remains to prove that $G$ admits a $d$-lucky labeling using $\left\lceil \frac{n+1}{2}\right\rceil$ labels. 

Let $V(K_n) = \{v_{_{1}}, \ldots, v_{_{n}}\}$, let $u_{_{1}}, \ldots, u_{_{n}}$ be the respective adjacent vertices of $G$ obtained by subdividing the matching edges between $K_n$ and $C_{m,n}$, and let $w_{_{1}}, \ldots, w_{_{n}}$ be the corresponding vertices of the top layer of $C_n$. Let $H$ be the subgraph of $G$ induced by the vertices $u_{_{1}}, \ldots, u_{_{n}}, v_{_{1}}, \ldots, v_{_{n}}$. 

Set $k = \left\lceil \frac{n+1}{2}\right\rceil$ and construct a labeling $\ell:V(G)\rightarrow [k]$ as follows. Let $\ell$ restricted to $H$ be the labeling of $K_n\circ K_1 = H$  with $\eta_{_{dl}}(H) = k$ labels from the proof of Theorem~\ref{thm:corona}. This labeling yields consecutive numbers $\frac{n^{2}+10n}{8} + j$, $j\in [n]$, as $d$-sums of vertices of $K_{n}$, when $n$ is even, and consecutive numbers $\frac{n^{2}+12n-5}{8} + j$, $j\in [n]$, as $d$-sums of vertices of $K_{n}$ when $n$ is odd. Next, for every $i\in [n]$ set $\ell(w_{i}) = 1$ when $m$ is odd,  and $\ell(w_{i}) = 2$ when $m$ is even. Next, label the remaining unlabeled vertices with labels $1$ and $2$ such that each vertex labeled 1 has all neighbors labeled $2$ and vice versa, every edge adjacent to vertex labeled $1$ as $2$ and vice versa. (Note that this is possible as $G-K_n$ is bipartite.) Finally, if $m$ is even and $n \geq 15$, redefine $\ell(w_{_{k+7}}) = 3$, and if $m$ is odd and $n\geq 8$, redefine $\ell(w_{_{5}}) = 3$ and $l(w_{_{k+3}}) = 3$. 

For every edge $e=uv$ of $G$, we have to prove that $d_\ell(u)\neq d_\ell(v)$ and consider typical edges. 

\medskip\noindent
{\bf Case 1}: $e = u_iv_i$, $i\in [n]$. \\
Since the degree of each vertex of $K_{n}$ in $G$ is $n$ and it is possible to label each vertex of $K_{n}$ as $1$, the $d$-sum of a vertex $v_i\in V(K_{_{n}})$, $i\in [n]$, is at least $(n-1)+1+n = 2n$ if $m$ is even and at least $(n-1)+2+n = 2n+1$ if $m$ is odd. In other words, for all $n$, the minimum $d$-sum of a vertex in $K_{_{n}}$ is $2n$. On the other hand, for all $n$, the maximum $d$-sum of a vertex $u_{_{i}}$, $i\in [n]$, is $k+5 = \left\lceil \frac{n+1}{2}\right\rceil + 5$. By a straightforward induction we can see that $\left\lceil \frac{n+1}{2}\right\rceil + 5 < 2n$ holds for $n\geq 5$. Hence the $d$-sums of the adjacent vertices $v_{_{i}}$ and $u_{_{i}}$ are distinct.

\medskip\noindent
{\bf Case 2}: $e = u_iw_i$, $i\in [n]$. \\
By our labeling, $k$ vertices, say $v_{_{1}},\ldots, v_{_{k}}$ are labeled $1$ and the corresponding adjacent vertices $u_{_{1}}, \ldots, u_{_{k}}$ are labeled $1,\ldots, k$. Further, $\ell(u_{_{i}}) = i$, $i\in [k]$, and $\ell\left(u_{_{k+i}}\right) = 1$, $i\in [n-k]$. 

Suppose first that $m$ is even. For $5\leq n \leq 14$, $\ell(w_{_{i}}) = 1$, $i\in [n]$, and for $n\geq 15$, 
$$\ell(w_{_{i}})=\left\{
\begin{array}{ll}
1; & i\in [n]\ {\rm and}\ i\neq k+7,\\
3, & i = k + 7.
\end{array}\right.$$
Hence when $5\leq n \leq 14$ we have 
$$d_\ell(u_{_{i}})=\left\{
\begin{array}{ll}
4; & i\in [k],\\
i - k + 4; & k+1\leq i\leq n,
\end{array}\right.$$
and 
$$d_\ell(w_{_{i}})=\left\{
\begin{array}{ll}
i+10; & i\in [k],\\
11; & k+1 \leq i\leq n.
\end{array}\right.$$
If $n\geq 15$, then we have 
$$~~~~~~~~d_\ell(u_{_{i}})=\left\{
\begin{array}{ll}
4; & i\in [k],\\
i-k+4; & k+1 \leq i\leq n\ {\rm and}\ i\neq k+7,\\
13; & i=k+7,\\
 \end{array}\right.$$
and 
$$d_\ell(w_{_{i}})=\left\{
\begin{array}{ll}
i+10; & i\in [k],\\
11, & k+1\leq i\leq n.\\
\end{array}\right.$$
Suppose now that $d_\ell(u_{_{i}}) = d_\ell(w_{_{i}})$. Since $n\ge 5$ this implies that $i - k + 4 = 11$, that is, $i = k + 7$, where $k + 1 \leq i \leq n$. Now $k + 7 \leq n$, that is, $\left\lceil \frac{n+1}{2}\right\rceil + 7\leq n$, yields $n\geq 15$. But then $d_\ell(u_{_{i}}) = 13 \neq 11 = d_\ell(w_{_{i}})$ for $i = k + 7$. This completes the argument when $m$ is even. 

Suppose next that $m$ is odd. For $5\leq n \leq 14$, $\ell(w_{_{i}}) = 2$, $i\in [n]$, and for $n\geq 15$, 
$$\ell(w_{_{i}})=\left\{
\begin{array}{ll}
2; & i\in [n]\ {\rm and}\ i\neq 5, k+3,\\
3; & i=5\ {\rm or}\ i=k+3.
\end{array}\right.$$
Therefore, if $5\leq n\leq 7$ we have
$$d_\ell(u_{_{i}})=\left\{
\begin{array}{ll}
5; & i\in [k],\\
i - k + 5; & k + 1\leq i\leq n,
\end{array}\right.$$
and 
$$d_\ell(w_{_{i}})=\left\{
\begin{array}{ll}
i + 7; & i\in [k],\\
8, & k + 1 \leq i \leq n.
\end{array}\right.$$
And if $n\geq 8$, then we have 
$$~~~~~~~d_\ell(u_{_{i}})=\left\{
\begin{array}{ll}
5; & i\in [4]\ {\rm or}\ 6 \leq i \leq k,\\
6; & i =5,\\
i -k + 5; & k + 1\leq i\leq n\ {\rm and}\ i \neq k+3,\\
9; & i=k+3,
\end{array}\right.$$
and 
$$d_\ell(w_{_{i}})=\left\{
\begin{array}{ll}
i + 7; & i\in [k], \\
8; & k + 1\leq i \leq n.
 \end{array}\right.$$
As in the case when $m$ was even, we can prove that for if $n\geq 5$, then $d_\ell(u_{_{i}}) \neq d_\ell(w_{_{i}})$ for $i\in [n]$. 

\medskip\noindent
{\bf Case 3}: $e$ is a radial edge. \\
We begin with the radial edges $w_{_{i}}x_{_{i}}$, $i\in [n]$, where $x_{_{i}}$ are vertices of layer $1$ of the subgraph $C_{_{m,n}}$ of $G$.

Suppose first that $m$ is even. Then $d_\ell(x_{_{i}}) = 8$ and $d_\ell(w_{_{i}})\geq 11$ for $i\in [n]$. It follows that the end vertices of the radial edges $w_{_{i}}x_{_{i}}$ receive distinct $d$-sums.

Assume next that $m$ is odd. Then
$$d_\ell(x_{_{i}})=\left\{
\begin{array}{ll}
12; & i\in [4]\ {\rm or}\ 6 \leq i \leq n\ {\rm or}\ i\leq k+3,\\
13, & i=5\ {\rm or}\ i=k+3.
\end{array}\right.$$
If $8 \leq d_\ell(w_{_{i}})\leq k+7$, then 
$$d_\ell(w_{_{i}})=\left\{
\begin{array}{ll}
7+i; & i\in [k],\\
8, & k+1\leq i \leq n.
 \end{array}\right.$$
If $i = 5$, then $d_\ell(w_{_{i}}) = 12 \neq 13 = d_\ell(x_{_{i}})$, and if $i \neq 5$, then $d_\ell(w_{_{i}}) \neq d_\ell(x_{_{i}})$. Thus the end vertices of the edges $e=w_{_{i}}x_{_{i}}$, $i \in \left[n\right]$ receive distinct $d$-sums.

 The end vertices of radial edges which are in layers $1, \ldots, m - 2$ receive $d$-sums $8$ and $12$, respectively. The radial edges with one end-vertex in layer $m - 1$ and the other end-vertex in the bottom layer, receive $d$-sums $8$ and $9$, respectively.

\medskip\noindent
{\bf Case 4}: $e$ is a cycle-subdivided edge. \\
If $v$ is a vertex subdividing a cycle edge, then, for all $m,n$ we have $\max d_\ell(v) = 6$, whereas $\min_{i\in [n]} d_\ell(w_{_{i}}) = 8$.
Hence the $d$-sums of end vertices of cycle-subdivided edges are also distinct.
\end{proof}

A {\em generalized cocktail-party graph} $H_{n,t}$, $n, t\ge 1$, is the complete $t$-partite graph with each partite set of size $n$, cf.~\cite{cvet1981}. If $n,t, r\ge 1$, set $H_{n,t,r} = H_{n,t}\circ \overline{K_{r}}$. That is, $H_{n,t,r}$ is obtained from $H_{n,t}$ by attaching $r$ pendant vertices at each of its vertices. 

\begin{thm}
\label{thm:cocktail}
If $n,r\ge 1$ and $t\ge 2$, then $\eta_{_{dl}}(H_{n,t,r}) = \left\lceil \frac{t + n + r - 1}{n + r}\right\rceil$.
\end{thm}

\begin{proof}
To simplify the notation, set $G = H_{n,t,r}$ for the rest of the proof. 

Let  $\ell: V(G) \rightarrow [k]$ be an optimal $d$-lucky labeling of $G$ and set 
$$x = \sum_{u\in V(H_{n,t})} \ell(u)\,.$$
The minimum $d$-sum and the maximum $d$-sum of a vertex of $G$ from $H_{n,t}$ are $r + r + \left(nt - n\right)+\left(x - nk\right)$ and $rk + r + \left(nt - n\right) + \left(x - n\right)$ respectively. Therefore the number of distinct $d$-sums of the vertices from from $H_{n,t}$ is at most 
$$\left[\left(rk + r\right) + \left(nt - n\right) + \left(x - n\right)\right] - \left[r + r + \left(nt - n\right)+\left(x - nk\right)\right] + 1\,,$$
which simplifies to $k(n + r) - (n + r) + 1$. Since the $d$-sums of vertices from $H_{n,t}$ that belong to different partite sets are different, this implies that $t \leq k(n + r) - (n + r) + 1$. Thus $\eta_{_{dl}}(G) = k \geq \left\lceil \frac{t + n + r - 1}{n + r}\right\rceil$.

For the other inequality, set $k = \left\lceil \frac{t + n + r - 1}{n + r}\right\rceil$, let $p = \left\lfloor \frac{t - (kr - r + 1)}{n}\right\rfloor$, and let $q = (t - (kr - r + 1)) \bmod n$. Let $V_{_{1}}, \ldots, V_t$ be the $t$-partite sets of $H_{_{n,t}}$. Let ${\cal W} = \{V_{_{1}}, \ldots, V_{_{kr - r + 1}}\}$,  let ${\cal U}_j = \{V_{_{kr - r + 1 + (j - 1)n + 1}}, \ldots, V_{_{kr - r + 1 + (j - 1)n + n}}\}$, $j\in [p]$, and let ${\cal U}_{p+1} = \{V_{_{kr - r + 1 + pn + 1}},\ldots, V_{_{t}}\}$. Note that $|{\cal W}| = kr - r + 1$, $|{\cal U}_j| = n$, and $|{\cal U}_{p + 1}| = q$. In particular, if $n$ divides $t -(kr -r + 1)$, then ${\cal U}_{p + 1} = \emptyset$. 

We now define $\ell: V(G) \rightarrow [k]$ as follows. 
\begin{itemize}
\item The vertices in all the parts of ${\cal W}$ are labeled $1$. 
\item For $j\in [p + 1]$ label all the vertices of the partite sets in ${\cal U}_j$ with $j+1$.  
\item Label every set of $r$ pendant vertices adjacent to each vertex of the partite set $V_i$ from ${\cal W}$ with equal label sequences $(\ell^{^{i}}_{_{1}},\ldots, \ell^{i}_r)$, $i\in  [kr - r + 1]$, such that $(\ell^{^{1}}_{_{1}},\ldots, \ell^{^{1}}_{_{r}}) = (1,\ldots, 1)$, and such that the Hamming distance between $(\ell^{^{i}}_{_{1}}, \ldots, \ell^{^{i}}_{_{r}})$ and $(\ell^{^{i+1}}_{_{1}}, \ldots, \ell^{^{i+1}}_{_{r}})$ is $1$ for $i \in [kr - r]$.
\item To label the rest of the pendant vertices in $G$, repeat the same procedure for each of the $n$ number of partite sets in ${\cal U}_j$, $j\in [p]$, as well as for the partite sets in ${\cal U}_{_{p+1}}$. See Figure \ref{h3,8,4}.
\end{itemize}

\begin{figure}[ht!]
\centering
\includegraphics[scale=0.70]{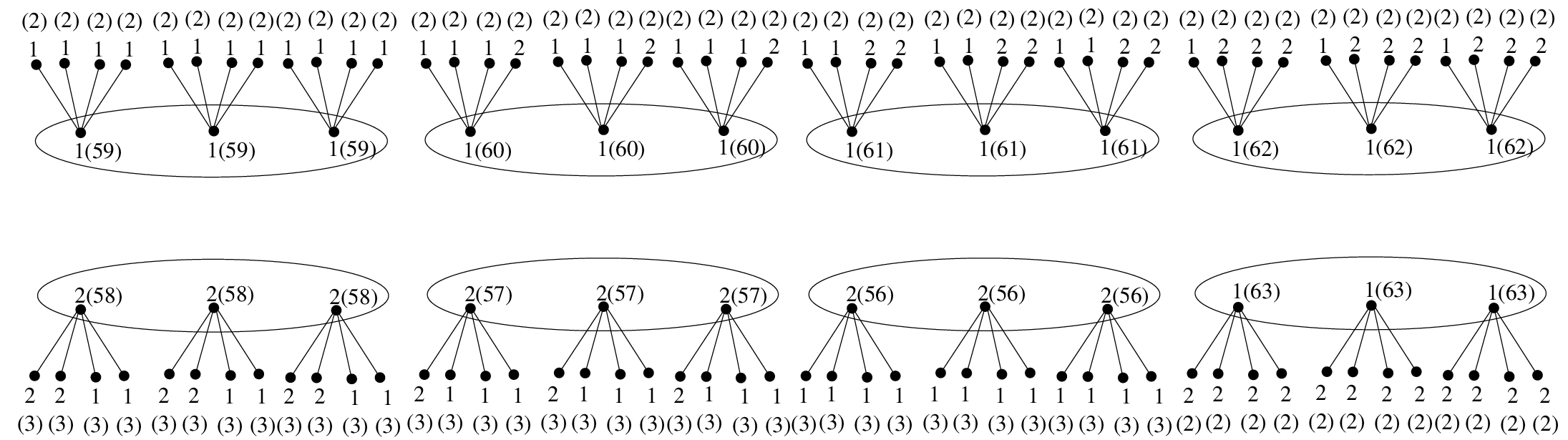} 
\caption{$d$-lucky labeling of $H_{3,8,4}$-generalized cocktail-party graph}
\label{h3,8,4}
\end{figure}

The $d$-sums of vertices in a partite set are all equal. Let $s_{_{i}}$ denote the $d$-sum of a representative vertex in $V_{_{i}}$, $i\in [t]$. Let $x$ be the sum of all labels of vertices in $H_{_{n,t}}$. Then $x - n + r, x - n + (r +1), \ldots,~x - n + kr$ are the $d$-sums of representative vertices in $V_{_{1}}, V_{_{2}}, \ldots, V_{_{kr-r+1}}$ respectively. Similarly, the $d$-sums of representative vertices in the $n$  partite sets in ${\cal U}_{_{j}}$ are $x - (j+1)n + r, x - (j+1)n + (r+1), \ldots, x - (j+1)n + (n + r -1) = x - jn + (r-1)$, $j \in [p]$. The same is true for ${\cal U}_{_{p+1}}$. This implies that $s_{_{t}}, s_{_{t-1}}, \ldots, s_{_{1}}$ are distinct consecutive integers. The $d$-sums of pendant vertices and vertices from the same partite set do not affect the $d$-lucky number. Further, the value of $k$ is optimum when $t$ is such that all vertices in $V_{_{t}}$ are labeled $k$ and $n$ divides $(t -(kr -r + 1))$. Then $t = kr - r+1 + \underbrace{n+\ldots+ n}_{(k-1)\ {\rm times}}$. We conclude that $\eta_{_{dl}}(G) \le k = \left\lceil \frac{t + n + r - 1}{n + r}\right\rceil$. 
\end{proof}

\section*{Acknowledgements}

Sandi Klav\v zar acknowledges the financial support from the Slovenian Research Agency (research core funding No.\ P1-0297 and project J1-9109).

\end{document}